\documentclass[a4paper]{amsart}
\author[Raghavan]{Dilip Raghavan}
\thanks{This paper was completed when the first author was a Fields Research Fellow.
The first author thanks the Fields Institute for its kind hospitality.}
\address[Raghavan]{Department of Mathematics \\
National University of Singapore\\
Singapore 119076.}
\email{\href{dilip.raghavan@protonmail.com}{dilip.raghavan@protonmail.com}}
\urladdr{\url{https://dilip-raghavan.github.io/}}
\author[Todorcevic]{Stevo Todorcevic}
\thanks{Second author is partially supported by grants from NSERC (455916) and CNRS (IMJ-PRG UMR7586).}
\address[Todorcevic]{Department of Mathematics, University of Toronto, Toronto, Canada, M5S 2E4.}
\email{\href{stevo@math.toronto.edu}{stevo@math.toronto.edu}}
\address[Todorcevic]{Institut de Math\'{e}matique de Jussieu, UMR 7586, Case 247, 4 place Jussieu, 75252 Paris Cedex, France.}
\email{\href{todorcevic@math.jussieu.fr}{todorcevic@math.jussieu.fr}}
\address[Todorcevic]{Matemati\v{c}ki Institut, SANU, Belgrade, Serbia.}
\email{\href{stevo.todorcevic@sanu.ac.rs}{stevo.todorcevic@sanu.ac.rs}}
\date{\today}
\title[Galvin's problem]{Galvin's problem in higher dimensions}
\usepackage[only, llbracket, rrbracket]{stmaryrd}
\usepackage{amssymb, amsmath, amsthm, mathrsfs, enumitem, amsfonts, latexsym, bbm}
\usepackage{hyperref}
\def\polhk#1{\setbox0=\hbox{#1}{\ooalign{\hidewidth
    \lower1.5ex\hbox{`}\hidewidth\crcr\unhbox0}}}
\newtheorem{Theorem}{Theorem}

\newtheorem{Lemma}[Theorem]{Lemma}
\newtheorem{Cor}[Theorem]{Corollary}

\newtheorem{Question}[Theorem]{Question}

\theoremstyle{definition}

\theoremstyle{remark}

\renewcommand{\[}{\left[}
\renewcommand{\]}{\right]}

\newcommand{\PPP}{\mathcal{P}}
\newcommand{\QQ}{\mathbb{Q}}
\newcommand{\lc}{\left|}
\newcommand{\rc}{\right|}

\newcommand\ZFC{\mathrm{ZFC}}

\newcommand\PFA{\mathrm{PFA}} 

\newcommand\CH{\mathrm{CH}}

\DeclareMathOperator{\otp}{otp}

\DeclareMathOperator{\dom}{dom}

\newcommand{\Pset}{\mathcal{P}}

\newcommand{\CCC}{{\mathcal{C}}}

\newcommand{\TTT}{{\mathcal{T}}}

\newcommand{\RR}{\mathbb{R}}

\newcommand{\pr}[2]{\left\langle #1, #2 \right\rangle}
\newcommand{\seq}[4]{\left\langle {#1}_{#2}: #2 #3 #4 \right\rangle}

\newcommand{\pc}[2]{{\[#1\]}^{#2}}
\begin{document}
\begin{abstract}
 It is proved that for each natural number $n$, if $\lc \RR \rc = {\aleph}_{n}$, then there is a coloring of $\pc{\RR}{n+2}$ into ${\aleph}_{0}$ colors that takes all colors on $\pc{X}{n+2}$ whenever $X$ is any set of reals which is homeomorphic to $\QQ$.
 This generalizes a theorem of Baumgartner and sheds further light on a problem of Galvin from the 1970s.
 Our result also complements and contrasts with our earlier result saying that any coloring of $\pc{\RR}{2}$ into finitely many colors can be reduced to at most $2$ colors on the pairs of some set of reals which is homeomorphic to $\QQ$ when large cardinals exist.
\end{abstract}
\maketitle
\section{Introduction} \label{sec:intro}
Let $A$ and $B$ be structures.
For natural numbers $k, l, t \geq 1$, the notation
\begin{align*}
 B \rightarrow {\left( A \right)}^{k}_{l, t}
\end{align*}
means that for every coloring $c: {\[B\]}^{k} \rightarrow l$, there exists a substructure $C$ of $B$ such that $C$ is isomorphic to $A$ and $\lc c'' {\[C\]}^{k} \rc \leq t$.
Suppose that $\CCC$ is some class of structures and that $A$ is a structure that embeds into every member of $\CCC$.
For a natural number $k \geq 1$, the \emph{$k$-dimensional Ramsey degree of $A$ within $\CCC$} is the the smallest natural number ${t}_{k} \geq 1$ (if it exists) such that $B \rightarrow {\left( A \right)}^{k}_{l, {t}_{k}}$, for every natural number $l \geq 1$ and for every structure $B \in \CCC$.
When no such ${t}_{k}$ exists, we say that the $k$-dimensional Ramsey degree of $A$ within $\CCC$ \emph{is infinite} or \emph{does not exist.}

Suppose that ${R}_{1}, \dotsc, {R}_{m}$ are finitely many finitary relations on the structure $A$.
The relations ${R}_{1}, \dotsc, {R}_{m}$ are said to \emph{solve the expansion problem for $A$ within the class $\CCC$} if for every structure $B \in \CCC$ and every finitary relation $S$ on $B$, there exists a substructure $C$ of $B$ and an isomorphism $\varphi: A \rightarrow C$ such that the restriction of $S$ to $C$ is quantifier free definable from the images of ${R}_{1}, \dotsc, {R}_{m}$ under $\varphi$.
It turns out that solving the expansion problem for $A$ within $\CCC$ for $k$-ary relations is equivalent to finding the $k$-dimensional Ramsey degree of $A$ within $\CCC$.

The notion of $k$-dimensional Ramsey degree discussed here is related to, but is distinct from other notions of Ramsey degree occurring frequently in the literature, such as in \cite{kpt, lionel, zukerdegree}.
The latter notions focus on the situation where $B$ is a Fra{\" i}ss{\' e} limit of finite structures, $A=B$, $\CCC = \{B\}$, and instead of coloring the $k$-element subsets of $B$, one colors all the copies of some fixed finite structure within $B$.
The computation of Ramsey degrees and the accompanying solution to expansion problems plays a crucial role in topological dynamics under the guise of determining the universal minimal flows of various automorphism groups.
We refer to \cite{kpt} for further details.

The focus of our work is on determining the $k$-dimensional Ramsey degrees of the topological space of the rationals within various different classes of Hausdorff spaces.
Let ${\TTT}_{\RR}$ denote the usual topology of the real numbers, and ${\TTT}_{\QQ}$ its restriction to the rationals.
In the 1970s, Galvin conjectured that the $2$-dimensional Ramsey degree of $\pr{\QQ}{{\TTT}_{\QQ}}$ within $\left\{ \pr{\RR}{{\TTT}_{\RR}} \right\}$ is $2$.
Galvin's conjecture was brought into sharper focus by Baumgartner who proved that the $2$-dimensional Ramsey degree of $\pr{\QQ}{{\TTT}_{\QQ}}$ does not exist within any class which contains a countable Hausdorff space.
\begin{Theorem}[Baumgartner~\cite{baumtop}] \label{thm:baum}
 Suppose $\pr{X}{\TTT}$ is any Hausdorff space with $\lc X \rc = {\aleph}_{0}$.
 There is a coloring $c: \pc{X}{2} \rightarrow \omega$ such that for any subspace $R \subseteq X$ that is homeomorphic to $\QQ$, $c''\pc{R}{2} = \omega$.
\end{Theorem}
An unpublished result of Todorcevic and Weiss extended Baumgartner's result to include $\sigma$-discrete metric spaces.
\begin{Theorem}[Todorcevic and Weiss~\cite{todorcevicweissnotes}] \label{thm:todorcevicweiss}
 If $\pr{X}{d}$ is a $\sigma$-discrete metric space, then there is a coloring $c: \pc{X}{2} \rightarrow \omega$ such that $c''\pc{Y}{2} = \omega$ for all $Y \subseteq X$ homeomorphic to $\QQ.$
\end{Theorem}
In 2018, the authors proved in~\cite{galvin-pi} that a $\sigma$-discrete metric space is the only impediment which can prevent the $2$-dimensional Ramsey degree of $\pr{\QQ}{{\TTT}_{\QQ}}$ from being equal to $2$ within a class of metric spaces.
The following summarizes the main results from \cite{galvin-pi}.
\begin{Theorem}[\cite{galvin-pi}] \label{thm:galvin}
 If there is a Woodin cardinal, then the $2$-dimensional Ramsey degree of $\pr{\QQ}{{\TTT}_{\QQ}}$ within the class of all uncountable sets of reals is $2$.
 If there is a proper class of Woodin cardinals, then the $2$-dimensional Ramsey degree of $\pr{\QQ}{{\TTT}_{\QQ}}$ within the class of all non-$\sigma$-discrete metric spaces is equal to $2$, and it is at most $2$ within the class of all regular, non-left-separated spaces with point-countable bases.
 Further, these conclusions hold under $\PFA$.
\end{Theorem}
Theorem \ref{thm:galvin} leads to a simple solution of the expansion problem in the case of binary relations for $\pr{\QQ}{{\TTT}_{\QQ}}$ within the class of metric spaces.
In fact for uncountable sets of real numbers, every binary relation is quantifier free definable from equality, the usual ordering of the reals, and an arbitrary well-ordering on a homeomorphic copy of $\QQ$.
The large cardinal hypothesis of Theorem \ref{thm:galvin} can be weakened to a local statement on the existence of inner models that contain Woodin cardinals and correctly compute sufficiently large fragments of the cumulative hierarchy.
We refer to \cite{galvin-pi} for further details.

The main result of this paper is that there is no direct generalization of Theorem \ref{thm:galvin} to dimensions $3$ and higher.
We will prove that if $\pr{X}{\TTT}$ is any Hausdorff space of cardinality ${\aleph}_{n}$, then there is a coloring of $\pc{X}{n+2}$ into ${\aleph}_{0}$ colors which realizes all colors on $\pc{R}{n+2}$ for any subspace $R \subseteq X$ that is homeomorphic to $\QQ$.
In other words, if a class $\CCC$ contains any Hausdorff space of size ${\aleph}_{n}$, then the $n+2$-dimensional Ramsey degree of $\pr{\QQ}{{\TTT}_{\QQ}}$ within $\CCC$ does not exist.
In particular, if $\CH$ holds, then the $3$-dimensional Ramsey degree of $\pr{\QQ}{{\TTT}_{\QQ}}$ within $\{\pr{\RR}{{\TTT}_{\RR}}\}$ does not exist, and if the $k$-dimensional Ramsey degree of $\pr{\QQ}{{\TTT}_{\QQ}}$ does exist within $\{\pr{\RR}{{\TTT}_{\RR}}\}$ for every $k$, then the continuum must be at least ${\aleph}_{\omega+1}$.
Our result here shows that Theorem \ref{thm:galvin} is sharp at least as far as the dimension is concerned.

Sierpi{\' n}ski's classical coloring from~\cite{Sierp} shows that the $k$-dimensional Ramsey degree of $\pr{\QQ}{{\TTT}_{\QQ}}$ within $\{\pr{\RR}{{\TTT}_{\RR}}\}$ is at least $k!(k-1)!$.
There is also a natural way to generalize Sierpi{\' n}ski's coloring to any metrizable space.
Details may be found in Section 3 of \cite{galvin-pi}.
In view of this and the result we prove here, the following question suggests itself.
\begin{Question} \label{Q:degree}
 What is the largest class of topological spaces within which the $k$-dimensional Ramsey degree of $\pr{\QQ}{{\TTT}_{\QQ}}$ is equal to $k!(k-1)!$?
\end{Question}
\section{Notation} \label{sec:notation}
Our set-theoretic notation is standard.
If $c$ is a function, then $\dom(c)$ denotes the domain of $c$, and if $X \subseteq \dom(c)$, then $c''X$ is the image of $X$ under $c$ -- that is, $c''X = \{c(x): x \in X\}$.
For any $A$, $\Pset(A)$ denotes the powerset of $A$.
When $\kappa$ is a cardinal, $\pc{X}{\kappa}$ is $\{A \subseteq X: \lc A \rc = \kappa\}$, and $\pc{X}{< \kappa}$ denotes $\{A \subseteq X: \lc A \rc < \kappa\}$.

If $d$ is a metric on $Y$, then ${B}_{d}(y, \epsilon)$ denotes $\{z \in Y: d(y, z) < \epsilon\}$, for all $y \in Y$ and $\epsilon \in \RR$.
A topological space $\pr{X}{\TTT}$ is \emph{dense-in-itself} if for each $x \in X$ and each open neighborhood $U$ of $x$, there exists $y \in U$ with $y \neq x$.
A theorem of Sierpi{\' n}ski (see \cite{engelking}) says that $\pr{X}{\TTT}$ is homeomorphic to $\QQ$ with its usual topology if and only if it is non-empty, countable, metrizable, and dense-in-itself.
\section{The main result} \label{sec:main}
We prove the main result in this section.
The proof is a natural generalization of Baumgartner's proof from \cite{baumtop} via the following well-known set mapping theorem of Kuratowski~\cite{kuratowski-alephs}, which is usually formulated in a slightly different way.
Much more information about set mappings in general, and Kuratowski's theorem in particular, may be found in \cite{partitionbible}.
\begin{Lemma}[Kuratowski~\cite{kuratowski-alephs}] \label{lem:kura}
 For each $n \in \omega$, there exists ${f}_{n}: \pc{{\omega}_{n}}{n+1} \rightarrow \pc{{\omega}_{n}}{< {\aleph}_{0}}$ such that:
 \begin{enumerate}
  \item
  $\forall s \in \pc{{\omega}_{n}}{n+1}\[{f}_{n}(s) \subseteq \max(s) \]$;
  \item
  $\forall t \in \pc{{\omega}_{n}}{n+2} \exists \alpha \in t\[\alpha < \max(t) \ \text{and} \ \alpha \in {f}_{n}(t \setminus \{\alpha\}) \]$.
 \end{enumerate}
\end{Lemma}
\begin{proof}
 The proof is by induction on $n \in \omega$.
 For $n=0$, define ${f}_{0}: \pc{\omega}{1} \rightarrow \pc{\omega}{< {\aleph}_{0}}$ by ${f}_{0}(\{m\}) = m \subseteq m = \max(\{m\})$.
 For (2), if $\{l, m\} \in \pc{\omega}{2}$ with $l < m$, then $l \in m = {f}_{0}(\{m\}) = {f}_{0}(\{l, m\} \setminus \{l\})$.
 Now assume that ${f}_{n}: \pc{{\omega}_{n}}{n+1} \rightarrow \pc{{\omega}_{n}}{< {\aleph}_{0}}$ satisfying (1) and (2) is given.
 Fix a sequence $\seq{e}{\gamma}{\in}{{\omega}_{n+1}}$ such that for each $\gamma \in {\omega}_{n+1}$ ${e}_{\gamma}: \gamma \rightarrow {\omega}_{n}$ is 1-1.
 Define ${f}_{n+1}: \pc{{\omega}_{n+1}}{n+2} \rightarrow \pc{{\omega}_{n+1}}{< {\aleph}_{0}}$ as follows.
 Given $s \in \pc{{\omega}_{n+1}}{n+2}$, let $\gamma = \max(s) \in s \subseteq {\omega}_{n+1}$.
 Then $s \setminus \{\gamma\} \subseteq \gamma$ and ${e}_{\gamma}''\left( s \setminus \{\gamma\} \right) \in \pc{{\omega}_{n}}{n+1}$.
 Thus we may define ${f}_{n+1}(s) = {e}^{-1}_{\gamma} \left( {f}_{n} \left( {e}_{\gamma}'' \left( s \setminus \{\gamma\} \right) \right) \right) \subseteq \gamma = \max(s)$.
 To verify (2), let $t \in \pc{{\omega}_{n+1}}{n+3}$ be given.
 Let $\gamma = \max(t) \in t \subseteq {\omega}_{n+1}$.
 Then $t \setminus \{\gamma\} \subseteq \gamma$ and $u = {e}_{\gamma}'' \left( t \setminus \{\gamma\} \right) \in \pc{{\omega}_{n}}{n+2}$.
 By the induction hypothesis, there exists $\alpha' \in u$ such that $\alpha' \in {f}_{n} \left( u \setminus \{\alpha'\} \right)$.
 Let $\alpha \in t \setminus \{\gamma\}$ be so that ${e}_{\gamma}(\alpha) = \alpha'$, and observe that $\alpha \in {e}^{-1}_{\gamma} \left( {f}_{n} \left( u \setminus \{\alpha'\} \right) \right)$.
 Note that $\alpha \in t$ and that $\alpha < \gamma = \max(t)$.
 Put $s = t \setminus \{\alpha\}$ and note that $\max(s) = \gamma$.
 Since $t\setminus\{\gamma\} \subseteq \gamma$, $\{\alpha\} \subseteq \gamma$, and ${e}_{\gamma}$ is a 1-1 function with domain $\gamma$, ${e}_{\gamma}''\left( s \setminus \{\gamma\} \right) = {e}_{\gamma}'' \left( \left( t \setminus \{\gamma\} \right) \setminus \{\alpha\} \right) = {e}_{\gamma}'' \left( t \setminus \{\gamma\} \right) \setminus {e}_{\gamma}''\{\alpha\} = u \setminus \{\alpha'\}$.
 Therefore by definition, ${f}_{n+1} \left( t \setminus \{\alpha\} \right) = {f}_{n+1}(s) = {e}^{-1}_{\gamma} \left( {f}_{n} \left( {e}_{\gamma}'' \left( s \setminus \{\gamma\} \right) \right) \right) = {e}^{-1}_{\gamma} \left( {f}_{n} \left( u \setminus \{\alpha'\} \right) \right) \ni \alpha$.
\end{proof}
The next lemma says that when searching within a well-ordered topological space for a homeomorphic copy of $\QQ$ on which some coloring is nicely behaved, one can concentrate only on the copies of $\QQ$ which have order type $\omega$ with respect to the well-ordering.
\begin{Lemma} \label{lem:famP}
 Let $\pr{X}{\TTT}$ be a topological space and $\delta$ an ordinal.
 Suppose that $\seq{x}{\alpha}{<}{\delta}$ is a 1--1 enumeration of all the points of $X$.
 Let $\PPP \subseteq {\[\delta\]}^{< {\aleph}_{1}}$ be a family such that:
 \begin{enumerate}
  \item
  $\PPP$ is \emph{hereditary}, that is, $\forall A \in \PPP \forall B \subseteq A \[B \in \PPP\]$;
  \item
  there exists $A \in \PPP$ such that the subspace $\{{x}_{\alpha}: \alpha \in A\}$ is homeomorphic to $\QQ$.
 \end{enumerate}
 Then
 \begin{align*}
  \omega = \min\left\{\otp(A): A \in \PPP \ \text{and the subspace} \ \{ {x}_{\alpha}: \alpha \in A \} \ \text{is homeomorphic to} \ \QQ \right\}.
 \end{align*}
\end{Lemma}
\begin{proof}
 The right hand side of the equation being asserted is well-defined due to item (2) of the hypotheses.
 Let $\xi$ denote the right hand side.
 Obviously, $\omega \leq \xi < {\omega}_{1}$.
 Fix $A \in \PPP$ such that the subspace $\{ {x}_{\alpha}: \alpha \in A \}$ is homeomorphic to $\QQ$ and $\xi = \otp(A)$.
 Recall that if $Y \subseteq X$ is a subspace which is homeomorphic to $\QQ$, and if $U \subseteq Y$ is open in $Y$ and if $U \neq \emptyset$, then $U$ is homeomorphic to $\QQ$.
 Let us apply this observation to $Y = \{ {x}_{\alpha}: \alpha \in A \}$.
 Thus for every $U \subseteq Y$ which is open in $Y$ and is non-empty, $B = \{ \alpha < \delta: {x}_{\alpha} \in U \} \subseteq A$, which means $B \in \PPP$ because $\PPP$ is hereditary, and $\{ {x}_{\alpha}: \alpha \in B \} = U$ is a subspace of $X$ that is homeomorphic to $\QQ$, which means $\otp(B) = \xi$ because of the minimality of $\xi$.
 Now, fix a metric $d$ on $Y$ that is compatible with the subspace topology on $Y$.
 It follows that for any $\alpha \in A$ and non-empty open subset $U$ of $Y$, $\{ \zeta < \delta: {x}_{\zeta} \in U \} \not\subseteq A \cap \alpha$, and so if ${x}_{\alpha} \notin U$, then $\exists \alpha < \beta < \delta \[ {x}_{\beta} \in U \ \text{and} \ \beta \in A \]$.
 Let $\{ \pr{{n}_{k}}{{l}_{k}}: 2 \leq k < \omega \}$ be a 1--1 enumeration of $\omega \times \omega$ so that ${n}_{k} < k-1$, for all $k < \omega$.
 Choose ${\alpha}_{0}, {\alpha}_{1} \in A$ with ${\alpha}_{0} < {\alpha}_{1}$.
 For $2 \leq k < \omega$, suppose that ${\alpha}_{0} < \dotsb < {\alpha}_{k-1}$ belonging to $A$ are given.
 As ${n}_{k} < k-1$, ${\alpha}_{{n}_{k}} < {\alpha}_{k-1}$, and so ${x}_{{\alpha}_{{n}_{k}}} \neq {x}_{{\alpha}_{k-1}}$.
 Let $\epsilon > 0$ be such that ${B}_{d}\left({x}_{{\alpha}_{{n}_{k}}}, \epsilon\right) \subseteq {B}_{d}\left({x}_{{\alpha}_{{n}_{k}}}, \frac{1}{{l}_{k}+1}\right)$ and ${x}_{{\alpha}_{k-1}} \notin {B}_{d}\left({x}_{{\alpha}_{{n}_{k}}}, \epsilon\right)$.
 Then there exists ${\alpha}_{k-1} < {\alpha}_{k} < \delta$ such that ${x}_{{\alpha}_{k}} \in  {B}_{d}\left({x}_{{\alpha}_{{n}_{k}}}, \epsilon\right)$ and ${\alpha}_{k} \in A$.
 Now $B = \{ {\alpha}_{k}: k < \omega \} \subseteq A$ and $\otp(B) = \omega$ by construction.
 Since $\PPP$ is hereditary, $B \in \PPP$.
 Also $\{ {x}_{\alpha}: \alpha \in B \}$, being a subspace of $Y$, is countable, metrizable, non-empty, and by construction, it is dense-in-itself.
 Therefore, the subspace $\{ {x}_{\alpha}: \alpha \in B \}$ is homeomorphic to $\QQ$, and therefore $\xi = \omega$.
\end{proof}
\begin{Theorem} \label{thm:main}
 Let $n \in \omega$.
 Let $\pr{X}{\TTT}$ be any Hausdorff space with $\lc X \rc = {\aleph}_{n}$.
 There is a coloring $c: {\[X\]}^{n+2} \rightarrow \omega$ such that for any subspace $R \subseteq X$ that is homeomorphic to $\QQ$, $c'' {\[R\]}^{n+2} = \omega$.
\end{Theorem}
\begin{proof}
 Fix ${f}_{n}: \pc{{\omega}_{n}}{n+1} \rightarrow \pc{{\omega}_{n}}{< {\aleph}_{0}}$ satisfying (1) and (2) of Lemma \ref{lem:kura}.
 Fix also a 1--1 enumeration $\seq{y}{\alpha}{<}{{\omega}_{n}}$ of all points of $X$.
 For any $s \in \pc{{\omega}_{n}}{n+1}$, define $O(s) = {f}_{n}(s) \setminus s$.
 By (1) of Lemma \ref{lem:kura}, $O(s)$ is a finite subset of $\max(s)$.
 Since $\pr{X}{\TTT}$ is Hausdorff, it is possible to find $\left\{ {N}_{s}({y}_{\xi}): \xi \in O(s) \cup s \right\}$ such that ${N}_{s}({y}_{\xi})$ is an open neighborhood of ${y}_{\xi}$, and for all $\xi \neq \xi'$ belonging to $O(s) \cup s$, ${N}_{s}({y}_{\xi}) \cap {N}_{s}({y}_{\xi'}) = \emptyset$.
 Observe that if $t \in \pc{{\omega}_{n}}{n+2}$ and if there exists some $\xi \in O(t \setminus \{\max(t)\})$ such that ${y}_{\max(t)} \in {N}_{t \setminus \{\max(t)\}}({y}_{\xi})$, then such $\xi$ must be unique.
 
 Define a function $h: {\[{\omega}_{n}\]}^{n+2} \cup \{\uparrow\} \rightarrow {\[{\omega}_{n}\]}^{n+2} \cup \{\uparrow\}$ as follows.
 First define $h(\uparrow) = \; \uparrow$.
 Next for any $t \in \pc{{\omega}_{n}}{n+2}$, define
 \begin{align*}
  h(t) = \begin{cases}
   \left( t \setminus \{\max(t)\} \right) \cup \{\xi\} \ &\text{if} \ \xi \in O(t \setminus \{\max(t)\}) \ \text{and}\\
   &{y}_{\max(t)} \in {N}_{t \setminus \{\max(t)\}}({y}_{\xi}),\\
   \\
   &\text{if} \ \text{there is no} \ \xi \in O(t \setminus \{\max(t)\})\\
   \uparrow \ &\text{such that} \ {y}_{\max(t)} \in {N}_{t \setminus \{\max(t)\}}({y}_{\xi}).
  \end{cases}
 \end{align*}
 First, note that since $O\left( t \setminus \{\max(t)\} \right) \cap \left( t \setminus \{\max(t)\} \right) = \emptyset$, $\xi \notin \left( t \setminus \{\max(t)\} \right)$, and so $\left( t \setminus \{\max(t)\} \right) \cup \{\xi\}$ does indeed belong to $\pc{{\omega}_{n}}{n+2}$.
 Next, note that since $O\left( t \setminus \{\max(t)\} \right) \subseteq \max\left( t \setminus \{\max(t)\} \right)$, $\max\left( \left( t \setminus \{\max(t)\} \right) \cup \{\xi\} \right) < \max(t)$.
 Now for any $t \in \pc{{\omega}_{n}}{n+2} \cup \{\uparrow\}$, define ${h}^{0}(t) = t$, and ${h}^{k+1}(t) = h({h}^{k}(t))$, for all $k \in \omega$.
 As observed earlier, when ${h}^{k+1}(t) \neq \; \uparrow$, then $\max({h}^{k+1}(t)) < \max({h}^{k}(t))$.
 It follows that for some $k \in \omega$, ${h}^{k}(t) = \; \uparrow$.
 Since ${h}^{0}(t) \neq \; \uparrow$ when $t \in \pc{{\omega}_{n}}{n+2}$, we may define $c: {\[X\]}^{n+2} \rightarrow \omega$ as follows.
 For any $w \in {\[X\]}^{n+2}$ there exists a unique ${t}_{w} \in \pc{{\omega}_{n}}{n+2}$ with $w = \{{y}_{\alpha}: \alpha \in {t}_{w}\}$.
 Define
 \begin{align*}
  c(w) = \min\left( \left\{ k \in \omega: {h}^{k+1}({t}_{w}) = \; \uparrow \right\} \right).
 \end{align*}
 We will verify that $c$ has the required property.
 Towards a contradiction, assume that there exists a subspace $R \subseteq X$ that is homeomorphic to $\QQ$ and $c''{\[R\]}^{n+2} \neq \omega$.
 Let $\PPP = \left\{ M \in {\[{\omega}_{n}\]}^{< {\aleph}_{1}}: c''{\[\{{y}_{\alpha}: \alpha \in M\}\]}^{n+2} \neq \omega \right\}$.
 If $M \in \PPP$ and $L \subseteq M$, then $\pc{\{{y}_{\alpha}: \alpha \in L\}}{n+2} \subseteq \pc{\{{y}_{\alpha}: \alpha \in M\}}{n+2}$ and so $c''{\[\{{y}_{\alpha}: \alpha \in L\}\]}^{n+2} \neq \omega$.
 Therefore $\PPP$ is hereditary and by hypothesis item (2) of Lemma \ref{lem:famP} is satisfied.
 Hence by Lemma \ref{lem:famP}, there exists $M \in \PPP$ with $\otp(M) = \omega$ so that the subspace $\{{y}_{\alpha}: \alpha \in M\}$ is homeomorphic to $\QQ$.
 
 We will derive a contradiction to the choice of $M$ by showing that for each $k \in \omega$, there exists $w \in \pc{\{{y}_{\alpha}: \alpha \in M\}}{n+2}$ with $c(w) = k$.
 For $k=0$, fix any $s \in \pc{M}{n+1}$.
 By definition, ${N}_{s}({y}_{\max(s)})$ is an open neighborhood of ${y}_{\max(s)}$.
 Since the subspace $\{{y}_{\alpha}: \alpha \in M\} \subseteq X$ is homeomorphic to $\QQ$ and since ${N}_{s}({y}_{\max(s)}) \cap \{{y}_{\alpha}: \alpha \in M\}$ is a non-empty open subset of $\{{y}_{\alpha}: \alpha \in M\}$, it must be infinite.
 Hence we may choose $\delta \in M$ such that $\max(s) < \delta$ and ${y}_{\delta} \in {N}_{s}({y}_{\max(s)})$.
 Let $t = s \cup \{\delta\} \in \pc{M}{n+2}$.
 Note that $\max(t) = \delta$ and $t \setminus \{\max(t)\} = s$.
 Further, if $\xi \in O(s)$, then $\xi \neq \max(s)$ and so ${N}_{s}({y}_{\xi}) \cap {N}_{s}({y}_{\max(s)}) = \emptyset$, whence ${y}_{\delta} \notin {N}_{s}({y}_{\xi})$.
 Therefore, $h(t) = \; \uparrow$.
 It follows that $c(w) = 0$, where $w = \{ {y}_{\alpha}: \alpha \in t \}$.
 
 Proceeding by induction, suppose that $k \in \omega$ and that for some $u \in \pc{M}{n+2}$, $c(\{{y}_{\alpha}: \alpha \in u\}) = k$.
 This means that ${h}^{k+1}(u) = \; \uparrow$, but ${h}^{k}(u) \neq \; \uparrow$.
 By (2) of Lemma \ref{lem:kura}, there exists $\alpha \in u$ with $\alpha \in {f}_{n}\left( u \setminus \{ \alpha \}\right)$.
 Define $s = u \setminus \{ \alpha \} \in \pc{M}{n+1} \subseteq \pc{{\omega}_{n}}{n+1}$.
 Note that $\alpha \in {f}_{n}\left( s \right) \setminus s = O(s)$.
 By definition, ${N}_{s}({y}_{\alpha})$ is an open neighborhood of ${y}_{\alpha}$.
 Hence ${y}_{\alpha} \in {N}_{s}({y}_{\alpha}) \cap \{{y}_{\beta}: \beta \in M\}$, and so ${N}_{s}({y}_{\alpha}) \cap \{{y}_{\beta}: \beta \in M\}$ is a non-empty open subset of the subspace $\{{y}_{\beta}: \beta \in M\} \subseteq X$, which is homeomorphic to $\QQ$.
 Therefore ${N}_{s}({y}_{\alpha}) \cap \{{y}_{\beta}: \beta \in M\}$ is an infinite set.
 Choose $\delta \in M$ such that $\delta > \max(s)$ and ${y}_{\delta} \in {N}_{s}({y}_{\alpha})$.
 Define $t = s \cup \{\delta\} \in \pc{M}{n+2}$.
 Note that $\max(t) = \delta$ and that $t \setminus \{\max(t)\} = s$.
 Therefore by definition, $h(t) = s \cup \{\alpha\} = \left( u \setminus \{ \alpha \} \right) \cup \{\alpha\} = u$ because $\{ \alpha \} \subseteq u$.
 It follows that ${h}^{k+2}(t) = \; \uparrow$, but ${h}^{k+1}(t) \neq \; \uparrow$, and hence that $c(\{{y}_{\beta}: \beta \in t\}) = k+1$.
 This concludes the induction and the proof.
\end{proof}
Observe that Baumgartner's Theorem \ref{thm:baum} is simply the special case of Theorem \ref{thm:main} when $n=0$.
We will now point out a few corollaries of Theorem \ref{thm:main}.
\begin{Cor} \label{cor:1}
 Let $n \in \omega$.
 Let $\pr{X}{\TTT}$ be any Hausdorff space with $\lc X \rc = {\aleph}_{n}$.
 For each natural number $l \geq 1$, there is a coloring $d: \pc{X}{n+2} \rightarrow l$ such that for any subspace $R \subseteq X$ that is homeomorphic to $\QQ$, $d''\pc{R}{n+2} = l$.
\end{Cor}
\begin{proof}
 Let $c: \pc{X}{n+2} \rightarrow \omega$ be the coloring from Theorem \ref{thm:main}.
 Define $d: \pc{X}{n+2} \rightarrow l$ so that for each $w \in \pc{X}{n+2}$, $d(w) = c(w) \mod l$.
 So $d(w) \in \{0, \dotsc, l-1\}$.
 Now apply Theorem \ref{thm:main}.
\end{proof}
\begin{Cor} \label{cor:2}
 Let $n \in \omega$.
 Suppose $\CCC$ is any class of topological spaces.
 If $\CCC$ contains any Hausdorff space of cardinality at most ${\aleph}_{n}$, then the $n+2$-dimensional Ramsey degree of $\pr{\QQ}{{\TTT}_{\QQ}}$ within $\CCC$ does not exist.
\end{Cor}
Corollary \ref{cor:2} says that in $\ZFC$ the $3$-dimensional Ramsey degree of $\pr{\QQ}{{\TTT}_{\QQ}}$ within the class of all sets of reals of cardinality ${\aleph}_{1}$ does not exist.
Thus the partition relation $\pr{\RR}{{\TTT}_{\RR}} \rightarrow {\left( \pr{\QQ}{{\TTT}_{\QQ}} \right)}^{k}_{l,  k!(k-1)!}$ serves as a gauge of the size of the continuum.
An intriguing feature here is that testing for a well-behaved countable substructure of $\RR$ can reveal something about the size of the continuum.
\begin{Cor} \label{cor:3}
 If $\pr{\RR}{{\TTT}_{\RR}} \rightarrow {\left( \pr{\QQ}{{\TTT}_{\QQ}} \right)}^{3}_{l,  12}$, for all $1 \leq l < \omega$, then $\CH$ fails.
 For any $k \geq 1$, if for every $1 \leq l < \omega$, $\pr{\RR}{{\TTT}_{\RR}} \rightarrow {\left( \pr{\QQ}{{\TTT}_{\QQ}} \right)}^{k}_{l,  k!(k-1)!}$, then $\lc \RR \rc \geq {\aleph}_{k-1}$.
 If the $k$-dimensional Ramsey degree of $\pr{\QQ}{{\TTT}_{\QQ}}$ in $\left\{ \pr{\RR}{{\TTT}_{\RR}} \right\}$ exists for every natural number $k \geq 1$, then ${2}^{{\aleph}_{0}} \geq {\aleph}_{\omega+1}$.
\end{Cor}
\def\polhk#1{\setbox0=\hbox{#1}{\ooalign{\hidewidth
  \lower1.5ex\hbox{`}\hidewidth\crcr\unhbox0}}}
\providecommand{\bysame}{\leavevmode\hbox to3em{\hrulefill}\thinspace}
\providecommand{\MR}{\relax\ifhmode\unskip\space\fi MR }
\providecommand{\MRhref}[2]{%
  \href{http://www.ams.org/mathscinet-getitem?mr=#1}{#2}
}
\providecommand{\href}[2]{#2}

\end{document}